\documentclass[12pt]{article}
\usepackage{amsfonts,amsmath,amssymb,amsthm,stmaryrd,hyperref}

\usepackage[T1]{fontenc}
\usepackage[ansinew]{inputenc}  
\usepackage{graphicx,cancel,xcolor,hyperref,graphicx,geometry}
\usepackage{graphicx,url,etoolbox}
\usepackage{graphicx}
\usepackage{tikz}
\usepackage{multicol}
\usetikzlibrary{patterns}
\usepackage{amssymb,subfigure,lineno,url}
\usepackage{float}
\newtheorem{theorem}{Theorem}
\newtheorem{corollary}{Corollary}
\newtheorem{definition}{Definition}
\newtheorem{lemma}{Lemma}
\newtheorem{proposition}{Proposition}
\newtheorem{remark}{Remark}

\newcommand\demo{\xqed{$\square$}}
\renewenvironment{proof}{{\bfseries \noindent Proof.}}{\demo}
\newcommand\xqed[1]{
\leavevmode\unskip\penalty9999 \hbox{}\nobreak\hfill
\quad\hbox{#1}}

\def\RN{\mathbb{R}^N}
\def\LL{\mathcal{L}}

\def\LL{\mathcal L}

\begin{document}
\title{Semilinear Equations Including the Mixed Operator}
\author{Alaa Ayoub\\ \href{alaayoub94@gmail.com}{alaayoub94@gmail.com}}
\maketitle
\tableofcontents
  \begin{abstract}
We study the local and global existence of solutions to a semilinear evolution equation driven by a mixed local-nonlocal operator of the form \( L = -\Delta + (-\Delta)^{\alpha/2} \), where \( 0 < \alpha < 2 \). The Cauchy problem under consideration is  
\begin{equation*}
\partial_t u + t^\beta L u = -h(t) u^p, \quad x \in \mathbb{R}^N, \quad t > 0,
\end{equation*}
with nonnegative initial data \( u(x, 0) = u_0(x) \). We establish the existence and uniqueness of local solutions in \( L^\infty(\mathbb{R}^N) \) using a contraction mapping argument. Furthermore, we analyze conditions for global existence, proving that solutions remain globally bounded in time under appropriate assumptions on the parameters \( \beta \), \( p \), and the function \( h(t) \).
\end{abstract}
\section{Preliminaries}\label{sec2}
Consider the problem
\begin{equation}\label{eq}
\left\{\begin{array}{ll}\partial_t u+t^{\beta}\mathcal{L} u=-h(t)u^p,&\qquad x\in\mathbb{R}^N,\,t>0,\\\\
 u(x,0)= u_0(x)\geq0,&\qquad x\in\mathbb{R}^N,\\
 \end{array}\right.
\end{equation}
where $p>1$, $\beta\geq0$, $h:(0,\infty)\to(0,\infty)$,  $h\in L^1_{loc}(0,\infty)$, and
$$u_0\in L^1(\mathbb{R}^N)\cap C_0(\mathbb{R}^N).$$
We first recall the definition and some properties related to the fractional Laplacian needed to prove our Theorem.
\begin{definition}\cite{Silvestre}\label{def1}
Let $u\in\mathcal{S}$  be the Schwartz space of rapidly decaying $C^\infty$ functions in $\mathbb{R}^N$ and $s \in (0,1)$. The fractional Laplacian $(-\Delta)^s$ in $\mathbb{R}^N$ is a non-local operator given by
\begin{eqnarray*}
 (-\Delta)^s v(x)&:=& C_{N,s}\,\, p.v.\int_{\mathbb{R}^N}\frac{v(x)- v(y)}{|x-y|^{N+2s}}\,dy\\\\
&=&\left\{\begin{array}{ll}
\displaystyle C_{N,s}\,\int_{\mathbb{R}^N}\frac{v(x)- v(y)}{|x-y|^{N+2s}}\,dy,&\quad\hbox{if}\,\,0<s<1/2,\\
{}\\
\displaystyle C_{N,s}\,\int_{\mathbb{R}^N}\frac{v(x)- v(y)-\nabla v(x)\cdotp(x-y)\mathcal{X}_{|x-y|<\delta}(y)}{|x-y|^{N+2s}}\,dy,\quad\forall\,\delta>0,&\quad\hbox{if}\,\,1/2\leq s<1,\\
\end{array}
\right.
\end{eqnarray*}
where $p.v.$ stands for Cauchy's principal value, and $C_{N,s}:= \frac{s\,4^s \Gamma(\frac{N}{2}+s)}{\pi^{\frac{N}{2}}\Gamma(1-s)}$.
\end{definition}

We are rarely going to use the fractional Laplacian operator in the Schwartz space; it can be extended to less regular functions as follows. For $s \in (0,1)$, $\varepsilon>0$, let
\begin{eqnarray*}
L_{s,\varepsilon}(\Omega)&:=&\left\{\begin{array}{ll}
\displaystyle L_s(\mathbb{R}^N)\cap C^{0,2s+\varepsilon}(\Omega)&\quad\hbox{if}\,\,0<s<1/2,\\
{}\\
\displaystyle L_s(\mathbb{R}^N)\cap C^{1,2s+\varepsilon-1}(\Omega),&\quad\hbox{if}\,\,1/2\leq s<1,\\
\end{array}
\right.
\end{eqnarray*}
where $\Omega$ is an open subset of $\mathbb{R}^N$, $C^{0,2s+\varepsilon}(\Omega)$ is the space of $2s+\varepsilon$- H\"{o}lder continuous functions on $\Omega$,  $C^{1,2s+\varepsilon-1}(\Omega)$ the space of functions of $C^1(\Omega)$ whose first partial
derivatives are H\"{o}lder continuous with exponent $2s+\varepsilon-1$, and
$$L_s(\mathbb{R}^N)=\left\{u:\mathbb{R}^N\rightarrow\mathbb{R}\quad\hbox{such that}\quad \int_{\mathbb{R}^N}\frac{u(x)}{1+|x|^{N+2s}}\,dx<\infty\right\}.$$

\begin{proposition}\label{Frac}\cite[Proposition~2.4]{Silvestre}${}$\\
Let $\Omega$ be an open subset of $\mathbb{R}^N$, $s \in (0,1)$, and $f\in L_{s,\varepsilon}(\Omega)$ for some $\varepsilon>0$. Then $(-\Delta)^sf$ is a continuous function in $\Omega$ and $(-\Delta)^sf(x)$ is given by the pointwise formulas of Definition \ref{def1} for every $x\in\Omega$.
\end{proposition} 
\noindent{\bf Remark:} A simple sufficient condition for function $f$ to satisfy the conditions in Proposition \ref{Frac} is that $f\in  L^1_{loc}(\mathbb{R}^N)\cap C^{2}(\Omega)$.\\

\begin{lemma}\label{lemma3}\cite{Bonforte}
Let $\langle x\rangle:=(1+|x|^2)^{1/2}$, $x\in\mathbb{R}^N$, $s \in (0,1]$, $d\geq 1$, and $q_0>N$. Then 
$$\langle x\rangle^{-q_0}\in  L^\infty(\mathbb{R}^N)\cap C^{\infty}(\mathbb{R}^N),\qquad\partial_x^2\langle x\rangle^{-q_0}\in L^\infty(\mathbb{R}^N),$$
 and 
$$
\left|(-\Delta)^s\langle x\rangle^{-q_0}\right|\lesssim \langle x\rangle^{-N-2s}.
$$
\end{lemma}

\begin{lemma}\label{lemma4}
Let $\psi$ be a smooth function satisfying $\partial_x^2\psi\in L^\infty(\mathbb{R}^N)$. For any $R>0$, let $\psi_R$ be a function defined by
$$ \psi_R(x):= \psi(R^{-1} x) \quad \text{ for all } x \in \mathbb{R}^N.$$
Then, $(-\Delta)^s (\psi_R)$,  $s \in (0,1]$,  satisfies the following scaling properties:
$$(-\Delta)^s \psi_R(x)= R^{-2s}(-\Delta)^s\psi(R^{-1} x), \quad \text{ for all } x \in \mathbb{R}^N. $$
\end{lemma}

\begin{lemma}\label{lemma5} Let $R>0$, $p>1$, $0<\alpha <2$, $N\geq1$, and $N<q_0<N+\alpha p$. Then, the following estimate holds
\begin{equation}\label{10}
\int_{\mathbb{R}^N}(\Phi_R(x))^{-1/(p-1)}\,\big|\mathcal{L}\Phi_R(x)\big|^{p/(p-1)}\, dx\lesssim R^{-\frac{2 p}{p-1}+N}+R^{-\frac{\alpha p}{p-1}+N},
\end{equation}
where $\Phi_R(x)=\langle {x}/{R}\rangle^{-q_0}=(1+|x/R|^2)^{-q_0/2}$.
\end{lemma}
\begin{proof}  Let $\tilde{x}=x/R$; by Lemma \ref{lemma4} we have $(-\Delta)^s\Phi_R(x)=R^{-2s}(-\Delta)^s\Phi_R(\tilde{x})$, for $s\in\{1,\alpha/2\}$. Therefore, using Lemma \ref{lemma3}, we conclude that
\begin{eqnarray*}
&{}&\int_{\mathbb{R}^N}(\Phi_R(x))^{-1/(p-1)}\,\big|\mathcal{L} \Phi_R(x)\big|^{p/(p-1)}\, dx\\
&{}&\lesssim\int_{\mathbb{R}^N}(\Phi_R(x))^{-1/(p-1)}\,\big|(-\Delta) \Phi_R(x)\big|^{p/(p-1)}\, dx+\int_{\mathbb{R}^N}(\Phi_R(x))^{-1/(p-1)}\,\big|(-\Delta)^{\alpha/2} \Phi_R(x)\big|^{p/(p-1)}\, dx\\
&{}&\lesssim R^{-\frac{2p}{p-1}+N}\int_{\mathbb{R}^N}\langle \tilde{x}\rangle^{\frac{q_0}{p-1}-\frac{(N+2)p}{p-1}}\, d \tilde{x}+R^{-\frac{\alpha p}{p-1}+N}\int_{\mathbb{R}^N}\langle \tilde{x}\rangle^{\frac{q_0}{p-1}-\frac{(N+\alpha)p}{p-1}}\, d \tilde{x}\\
&{}&\lesssim R^{-\frac{2p}{p-1}+N}+R^{-\frac{\alpha p}{p-1}+N},
\end{eqnarray*}
where we have used the fact that $q_0<N+\alpha p<N+2p$ implies $\frac{(N+2)p}{p-1}-\frac{q_0}{p-1}>N$ and $\frac{(N+\alpha)p}{p-1}-\frac{q_0}{p-1}>N$.
\end{proof}

In order to introduce the notion of  the mild solution of our problem, and to prove the main results especially the theorem of local existence, we need to give a full review about the heat semigroup and the associated heat kernel of  the operator $-\mathcal{L}=\Delta-(-\Delta)^{\alpha/2}$, $\alpha\in (0,2)$ (see e.g. \cite{Song}). The comparison principle is needed as well.\\
We start by an alternative expression of the fractional Laplacian operator $(-\Delta)^{\alpha/2}$ via the Fourier transform 
$$
\begin{array}{llll}
(-\Delta)^{\alpha/2}:&H^\alpha(\mathbb{R}^N)\subseteq L^2(\mathbb{R}^N)&\longrightarrow& L^2(\mathbb{R}^N)\\
{}&\qquad \quad u&\longmapsto& (-\Delta)^{\alpha/2}(u)=\mathcal{F}^{-1}(|\xi|^\alpha\mathcal{F}(u))\\
\end{array}
$$
where 
$$H^\alpha(\mathbb{R}^N)=D\left((-\Delta)^{\alpha/2}\right)=\{u\in L^2(\mathbb{R}^N);\,\, |\xi|^\alpha\mathcal{F}(u)\in L^2(\mathbb{R}^N)\},\qquad s>0.$$
 As a consequence of the fact that $(-\Delta)^{\alpha/2}$ is a positive definite self-adjoint operator on the Hilbert space $L^2(\mathbb{R}^N)$, is that the operator $-(-\Delta)^{\alpha/2}$, e.g. Yosida \cite{Yosi} or \cite[Theorem~4.9]{Grigoryan}, generates a strongly continuous semigroup $T(t):=e^{-t(-\Delta)^{\alpha/2}}$ on $L^2(\mathbb{R}^N)$. It holds $T(t)v=P_\alpha(t)\ast v$, where $P_\alpha$ is
the fundamental solution of the fractional diffusion equation $u_t+(-\Delta)^{\alpha/2}u=0$, represented via the Fourier transform by
\begin{equation}\label{FS}
P_\alpha(t)(x):=P_\alpha(x,t)=\frac{1}{(2\pi)^{N/2}}\int_{\mathbb{R}^N}e^{ix.\xi-t|\xi|^\alpha}\,d\xi.
\end{equation}
It is well-known that this function, also known  as the heat kernel of $-(-\Delta)^{\alpha/2}$, satisfies
\begin{equation}\label{P_1}
    P_\alpha(1)\in L^\infty(\mathbb{R}^N)\cap
L^1(\mathbb{R}^N),\quad
P_\alpha(x,t)\geq0,\quad\int_{\mathbb{R}^N}P_\alpha(x,t)\,dx=1,
\end{equation}
\noindent for all $x\in\mathbb{R}^N$ and $t>0.$ Hence, using Young's inequality for the convolution
and the following self-similar form $P_\alpha(x,t)=t^{-N/\alpha}P_\alpha(xt^{-1/\alpha},1)$, see e.g. \cite{Miao}, we have
\begin{equation}\label{P_2}
\|P_\alpha(t)\ast v\|_q\;\leq \;Ct^{-\frac{N}{\alpha}(\frac{1}{r}-\frac{1}{q})}\|v\|_r,
\end{equation}
\begin{equation}\label{P_3}
\|\nabla P_\alpha(t)\|_q\;\leq \;Ct^{-\frac{N}{\alpha}(1-\frac{1}{q})-\frac{1}{\alpha}},
\end{equation}
for all $v\in L^r(\mathbb{R}^N)$ and  all $1\leq r\leq q\leq\infty,$ $t>0$.\\
Using $H^2(\mathbb{R}^N)\subseteq H^\alpha(\mathbb{R}^N)$, one can define the mixed local-nonlocal operator by
$$
\begin{array}{llll}
\mathcal{L}:&H^2(\mathbb{R}^N)\subseteq L^2(\mathbb{R}^N)&\longrightarrow& L^2(\mathbb{R}^N)\\
{}&\qquad \quad u&\longmapsto& \mathcal{L}(u)=-\Delta u+(-\Delta)^{\alpha/2}(u)\\
\end{array}
$$
which, naturally, implies that the operator $-\mathcal{L}$ generates a strongly continuous semigroup of contractions $\{S(t):=e^{-t\mathcal{L}}\}_{t\geq 0}$ on $L^2(\mathbb{R}^N)$. It holds $S(t)v=E_\alpha(t)\ast v$, where $E_\alpha(t)=P_2(t)\ast P_\alpha(t)$ is the heat kernel of the operator $-\mathcal{L}$, and $P_2(t)$ is the fundamental solution of the heat  equation $u_t-\Delta u=0$, represented by
$$P_2(t)=P_2(t,x)=\frac{1}{(4\pi t)^{N/2}}e^{-\frac{|x|^2}{4t}}.$$
Note that,  see \cite[Proposition~48.4]{souplet}, the function $P_2(t)$ is the heat kernel of $\Delta$ and satisfies the properties \eqref{P_1}, \eqref{P_2}, and \eqref{P_3} by replacing $\alpha$ by $2$.\\
It is clear that $E_\alpha(t)$ is the fundamental solution of the diffusion equation $u_t + \mathcal{L}u=0$, and satisfies the following properties.

\begin{lemma}\label{Property1}
The heat kernel $E_\alpha(t)$ satisfies the following properties
\begin{enumerate}
  \item[$(i)$] $E_\alpha \in C^{\infty}(\mathbb{R}^+\times\mathbb{R}^N)$ and $E_\alpha(t)\geq 0$, for all $x\in\mathbb{R}^N$ and $t>0$,
  \item[$(ii)$] For all $t>0$, we have 
  $$\int_{\mathbb{R}^N}E_\alpha(t,x)dx=1.$$
   \item[$(iii)$] For every $t,\tau>0$, we have 
  $$E_\alpha(t)\ast E_\alpha(\tau)=E_\alpha(t+\tau).$$
  \end{enumerate}
\end{lemma}
\begin{lemma}\label{Property2}
The heat kernel $E_\alpha(t)$ satisfies the following properties
\begin{itemize}
   \item For every $v\in L^r(\mathbb{R}^N)$ and  all $1\leq r\leq q\leq\infty$, $t>0$, we have
  \begin{equation}\label{Pr1}
  \|E_\alpha(t)\ast v\|_q\,\leq \;C\min\left\{t^{-\frac{N}{2}(\frac{1}{r}-\frac{1}{q})},\,t^{-\frac{N}{\alpha}(\frac{1}{r}-\frac{1}{q})}\right\}\|v\|_r.
\end{equation}
   \item For all $1\leq q\leq \infty$, we have 
    \begin{equation}\label{Pr2}
    \|\nabla E_\alpha(t)\|_q\,\leq \;C\min\left\{t^{-\frac{N}{2}(1-\frac{1}{q})-\frac{1}{2}},\,t^{-\frac{N}{\alpha}(1-\frac{1}{q})-\frac{1}{\alpha}}\right\}.
  \end{equation}
  \end{itemize}
\end{lemma}
\begin{proof} For every $1\leq r\leq q\leq\infty$ and $t>0$, using Young's inequality for the convolution and the fact that $\|P_i(t)\|_1=1$, $i=2,\alpha$, we have
\begin{eqnarray}\label{Pr4}
\|E_\alpha(t) \ast v\|_q&=&\|P_2(t) \ast \left(P_\alpha(t)\ast v\right)\|_q\nonumber\\
&\leq&C\,t^{-\frac{N}{2}(\frac{1}{r}-\frac{1}{q})}\|P_\alpha(t)\ast v\|_r\nonumber\\
&\leq&C\,t^{-\frac{N}{2}(\frac{1}{r}-\frac{1}{q})}\|P_\alpha(t)\|_1\|v\|_r\nonumber\\
&=&C\,t^{-\frac{N}{2}(\frac{1}{r}-\frac{1}{q})}\|v\|_r,
\end{eqnarray}           
  and
   \begin{eqnarray}\label{Pr5}
\|E_\alpha(t) \ast v\|_q&=&\|P_\alpha(t) \ast \left(P_2(t)\ast v\right)\|_q\nonumber\\
&\leq&C\,t^{-\frac{N}{\alpha}(\frac{1}{r}-\frac{1}{q})}\|P_2(t)\ast v\|_r\nonumber\\
&\leq&C\,t^{-\frac{N}{\alpha}(\frac{1}{r}-\frac{1}{q})}\|P_2(t)\|_1\|v\|_r\nonumber\\
&=&C\,t^{-\frac{N}{\alpha}(\frac{1}{r}-\frac{1}{q})}\|v\|_r.
\end{eqnarray}                
Combining \eqref{Pr4} and \eqref{Pr5}, \eqref{Pr1} is obtained. Similarly, to get \eqref{Pr2}, we have
$$
\|\nabla E_\alpha(t)\|_q=\|\nabla(P_2(t)) \ast P_\alpha(t)\|_q\leq\|\nabla(P_2(t))\|_q \|P_\alpha(t)\|_1\leq C\,t^{-\frac{N}{2}(1-\frac{1}{q})-\frac{1}{2}},
$$            
  and
  $$
\|\nabla E_\alpha(t)\|_q=\|\nabla(P_\alpha(t)) \ast P_2(t)\|_q\leq\|\nabla(P_\alpha(t))\|_q \|P_2(t)\|_1\leq C\,t^{-\frac{N}{\alpha}(1-\frac{1}{q})-\frac{1}{\alpha}}.
$$
\end{proof}
\begin{lemma}\label{Taylor}
Let $g\in L^1(\mathbb{R}^N)$ and put $\displaystyle M_g=\int_{\mathbb{R}^N}g(x)\,dx$. We have 
\begin{equation}\label{TaylorInequality1}
\lim\limits_{t\to\infty}\|E_\alpha(t)\ast g-M_gE_\alpha(t)\|_1=0.
\end{equation}
If, in addition, $xg(x) \in L^1(\mathbb{R}^N)$, then
\begin{equation}\label{TaylorInequality2}
\|E_\alpha(t)\ast g-M_gE_\alpha(t)\|_1\leq C\min\{t^{-1/2},\,t^{-1/\alpha}\}\|xg(x)\|_1,\qquad\hbox{for all}\,\,t>0.
\end{equation}
\end{lemma}
\begin{proof}
We adapt the technique used in \cite[Proposition~48.6]{souplet}. We first establish \eqref{TaylorInequality2} by supposing $g\in L^1(\mathbb{R}^N,\,(1+|x|)\,dx)$. Using Taylor's expansion and Fubini's theorem, we have
\begin{eqnarray*}
\|E_\alpha(t)\ast g-M_gE_\alpha(t)\|_1&=&\left\|\int_{\mathbb{R}^N}\left(E_\alpha(t,x-y)-E_\alpha(t,x)\right)g(y)\,dy\right\|_1\\
&=&\left\|\int_0^1\int_{\mathbb{R}^N}\nabla E_\alpha(t,x-\theta y) yg(y)\,dy\,d\theta\right\|_1\\
&\leq&\int_0^1\int_{\mathbb{R}^N}\left\|\nabla E_\alpha(t,x-\theta y)\right\|_1 yg(y)\,dy\,d\theta\\
&\leq&C\min\left\{t^{-\frac{1}{2}},\,t^{-\frac{1}{\alpha}}\right\}\int_0^1\int_{\mathbb{R}^N}yg(y)\,dy\,d\theta\\
&\leq&C\min\{t^{-1/2},\,t^{-1/\alpha}\}\|xg(x)\|_1,
\end{eqnarray*}
where we have use Minkowski's inequality and \eqref{Pr2}. Let us next prove \eqref{TaylorInequality1}; fix $g\in L^1(\mathbb{R}^N)$ and pick a sequence $\{g_j\}\in \mathcal{D}(\mathbb{R}^N)$ such that $g_j\rightarrow g$ in $L^1(\mathbb{R}^N)$. For each $j$, using the fact that $\|E_\alpha(t)\|_1=1$, we have
\begin{eqnarray*}
\|E_\alpha(t)\ast g-M_gE_\alpha(t)\|_1&\leq&\|E_\alpha(t)\ast g-E_\alpha(t)\ast g_j\|_1+\|E_\alpha(t)\ast g_j-M_{g_j}E_\alpha(t)\|_1+\|M_{g_j}E_\alpha(t)-M_gE_\alpha(t)\|_1\\
&\leq&\|g-g_j\|_1\|E_\alpha(t)\|_1+\|E_\alpha(t)\ast g_j-M_{g_j}E_\alpha(t)\|_1+|M_{g_j}-M_g|\|E_\alpha(t)\|_1\\
&\leq&2\|g-g_j\|_1+C\min\{t^{-1/2},\,t^{-1/\alpha}\}\|xg_j(x)\|_1.
\end{eqnarray*}
By \eqref{TaylorInequality2}, it follows that
$$\limsup_{t\rightarrow\infty}\|E_\alpha(t)\ast g-M_gE_\alpha(t)\|_1\leq 2\|g-g_j\|_1,$$
and the conclusion follows by letting $j\rightarrow\infty$.
\end{proof}
Before we end up the section, an important lemma will be used to provide  the fundamental solution of the diffusion equation $u_t + t^{\beta}\mathcal{L}u=0$.
\begin{lemma}\label{Property3}
Let us consider the homogeneous system
$$
\left\{\begin{array}{ll}
\displaystyle \partial_t u+t^{\beta}\mathcal{L} u=0,&\quad x\in\mathbb{R}^N,t>0,\\\\
 u(x,0) = u_0(x),&\quad x\in\mathbb{R}^N.\\
\end{array}\right.
$$
For every $t\geq t_0\geq 0$, it follows that
$$u(t)=E_\alpha\left(\frac{t^{\beta+1}}{\beta+1}-\frac{t_0^{\beta+1}}{\beta+1}\right)\ast u(t_0).$$
\end{lemma}
\begin{proof} Multiply the homogeneous equation by $t^{-\beta}$, we get $t^{-\beta}\partial_t u+\mathcal{L} u=0$. Let $v(x,t):=u(x,t+t_0)$, then
$$(t+t_0)^{-\beta}\partial_t v(x,t)+\mathcal{L} v(x,t)=0.$$
By considering the change of variable
$$\tau=\frac{(t+t_0)^{\beta+1}}{\beta+1}-\frac{t_0^{\beta+1}}{\beta+1},$$
and denoting
$$\widetilde{v}(x,\tau)=v(x,t),$$
we obtain
$$\partial_\tau \widetilde{v}(x,\tau)+\mathcal{L}\widetilde{v}(x,\tau)=0,$$
which yields to
$$\widetilde{v}(x,\tau)=E_\alpha(\tau)\ast \widetilde{v}(x,0),$$
i.e.
$$v(x,t)=E_\alpha\left(\frac{(t+t_0)^{\beta+1}}{\beta+1}-\frac{t_0^{\beta+1}}{\beta+1} \right)\ast v(x,0).$$
As $v(x,0)=u(x,t_0)$, we conclude that 
$$u(x,t)=v(x,t-t_0)=E_\alpha\left(\frac{t^{\beta+1}}{\beta+1}-\frac{t_0^{\beta+1}}{\beta+1} \right)\ast u(x,t_0).$$
\end{proof}
\begin{definition}[Mild solution]
Let $u_0\in  C_0(\mathbb{R}^N)$, $\alpha\in(0,2)$, $\beta\geq 0$, $p>1$, and $T>0$. We say that $u\in C([0,T),C_0(\mathbb{R}^N))$
is a mild solution of problem \eqref{eq} if $u$ satisfies the following integral equation
\begin{equation}\label{IE}
    u(t)=E_\alpha\left(\frac{t^{\beta+1}}{\beta+1}\right)\ast u_0(x)-\int_{0}^th(s) E_\alpha\left(\frac{t^{\beta+1}}{\beta+1}-\frac{s^{\beta+1}}{\beta+1} \right)\ast |u|^{p-1}u(x,s)\,ds,\quad t\in[0,T).
\end{equation}
More general, for all $0\leq t_0\leq t<T$, we have
\begin{equation}\label{IEG}
    u(t)=E_\alpha\left(\frac{t^{\beta+1}}{\beta+1}-\frac{t_0^{\beta+1}}{\beta+1} \right)\ast u(x,t_0)-\int_{t_0}^th(s) E_\alpha\left(\frac{t^{\beta+1}}{\beta+1}-\frac{s^{\beta+1}}{\beta+1} \right)\ast |u|^{p-1}u(x,s)\,ds.
\end{equation}
\end{definition}
We refer the reader to \cite{CH,6} to get the existence, the uniqueness and the regularity of mild solution of \eqref{eq}.
\begin{theorem}[Global existence]\label{T0}
Given $0\leq u_0\in  L^1(\mathbb{R}^N)\cap C_0(\mathbb{R}^N)$, $\alpha\in(0,2)$, $\beta\geq 0$, and $p>1$. Then, problem \eqref{eq} has a unique global mild solution 
$$u\in C([0,\infty),L^1(\mathbb{R}^N)\cap C_0(\mathbb{R}^N))\cap C^1((0,\infty),L^2(\mathbb{R}^N))\cap C((0,\infty),H^2(\mathbb{R}^N)).$$
\end{theorem}

\begin{lemma}[Nonnegativity]\label{nonnegativity}
Let $T>0$. If $u$ is a mild solution of problem \eqref{eq} on $[0,T)$, and $u_0\geq 0$, then $u(x,t)\geq 0$ for almost everywhere $x\in \mathbb{R}^N$ and for all $t\in [0,T)$.
\end{lemma}
\begin{proof}
Our goal is to prove $u^-=0$, where $u=u^+-u^-$, $u^+=\max(u,0)$, and $u^-=\max(-u,0)$. Multiplying the first equation of the system \eqref{eq} by $u^-$ and integrating over $\mathbb{R}^N$, we obtain
\begin{equation*}
\int_{\mathbb{R}^N}u_tu^- \,dx=\int_{\mathbb{R}^N}t^\beta\Delta(u)u^-\,dx -\int_{\mathbb{R}^N}t^\beta(-\Delta)^{\alpha/2}(u) u^-\,dx - \int_{\mathbb{R}^N}h(t)|u|^{p-1}u u^-\,dx.
\end{equation*}
Using the identities $u^+\, u^-=\Delta(u^+)u^-=0$, and $(-\Delta)^{\alpha/2}(u^+)u^-\leq 0$ almost everywhere, we get
\begin{equation}
-\int_{\mathbb{R}^N}u^{-}_t u^- \,dx\geq-\int_{\mathbb{R}^N}t^\beta\Delta(u^-)u^- \,dx +\int_{\mathbb{R}^N}t^\beta(-\Delta)^{\alpha/2}(u^-) u^- \,dx + \int_{\mathbb{R}^N}h(t)|u|^{p-1}(u^-)^2\,dx.
\end{equation}
Thus
\begin{equation}\label{MP1}
\frac{1}{2}\frac{d}{dt}\int_{\mathbb{R}^N}(u^{-})^2 \,dx\leq t^\beta\int_{\mathbb{R}^N}\Delta(u^-)u^- \,dx -t^\beta\int_{\mathbb{R}^N}(-\Delta)^{\alpha/2}(u^-) u^- \,dx\\
 -h(t) \int_{\mathbb{R}^N}|u|^{p-1}(u^-)^2\,dx.
\end{equation}
Applying Green's theorem, we have
\begin{equation}\label{MP2}
\int_{\mathbb{R}^N}\Delta(u^-)u^- \,dx=-\int_{\mathbb{R}^N}|\nabla(u^-)|^2\,dx,
\end{equation} 
and, using the self-adjoint property, we obtain
\begin{equation}\label{MP3}
\int_{\mathbb{R}^N}u^- (-\Delta)^{\alpha/2}u^-\,dx=\int_{\mathbb{R}^N}[(-\Delta)^{\alpha/4}(u^-)] ^2 \,dx.
\end{equation}
Inserting \eqref{MP2} and \eqref{MP3} into \eqref{MP1}, we infer that
\begin{equation}\label{MP4}
\frac{1}{2}\frac{d}{dt}\int_{\mathbb{R}^N}(u^{-})^2 \,dx\leq 0.
\end{equation}
Integrating \eqref{MP4} with respect to time, we arrive at 
\begin{equation}
\int_{\mathbb{R}^N}(u^{-}(x,t))^2 \,dx \leq \int_{\mathbb{R}^N}(u^{-}_0(x))^2 \,dx=0,\quad\hbox{for all}\,\, t\geq 0,
\end{equation}
which implies that $u^-=0$ a.e. $x\in\mathbb{R}^N$, for all $t\in[0,T)$,  and hence, $u=u^+\geq 0$.
\end{proof}
\begin{lemma}[The comparison principle]\label{Comparison}
Let $T>0$, and let $u$ and $v$, respectively, be mild solutions of problem \eqref{eq} with initial data $u_0$ and $v_0$, respectively. If $0\leq u_0\leq v_0$, then $0\leq u(x,t)\leq v(x,t)$ for almost every $x\in \mathbb{R}^N$ and for all $t\in [0,T)$.
\end{lemma}
\begin{proof}
Let $w(x,t)=v(x,t)-u(x,t)$, then
$$
\left\{\begin{array}{ll}
w_t=-t^\beta\mathcal{L} w-h(t)(v^p-u^p),&\quad x\in\mathbb{R}^N,\,t>0,\\
w(x,0)=w_0(x)\geq 0,&\quad x\in\mathbb{R}^N.
\end{array}
\right.
$$
Using the following estimation
$$0\leq v^p-u^p\leq C(v-u)(v^{p-1}+u^{p-1})\leq C (\|v\|_\infty^{p-1}+\|u\|_\infty^{p-1}) w,$$
we arrive at
$$
\left\{\begin{array}{ll}
w_t\geq -t^\beta\mathcal{L} w-C\,h(t)(\|v\|_\infty^{p-1}+\|u\|_\infty^{p-1}) w,&\quad x\in\mathbb{R}^N,\,t>0,\\
w(x,0)=w_0(x)\geq 0,&\quad x\in\mathbb{R}^N.
\end{array}
\right.
$$
Applying similar calculations as in the proof of Lemma \ref{nonnegativity}, we conclude that $w=w^+\geq 0$ a.e. $x\in\mathbb{R}^N$, for all $t\in [0,T)$,  and therefore $v\geq u$. This completes the proof of Lemma \ref{Comparison}.
\end{proof}
\section{A result of Local Existence}
Consider the problem
\begin{equation}\label{S1}
\left\{\begin{array}{ll}\partial_t u+t^{\beta}\mathcal{L} u=-h(t)u^p,&\qquad x\in\mathbb{R}^N,\,t>0,\\\\
 u(x,0)= u_0(x)\geq0,&\qquad x\in\mathbb{R}^N,\\
 \end{array}\right.
\end{equation}
where $p>1$, $\beta\geq0$, $h:(0,\infty)\to(0,\infty)$,  $h\in L^1_{loc}(0,\infty)$, and
$$u_0\in L^1(\mathbb{R}^N)\cap C_0(\mathbb{R}^N).$$
We will study in this chapter the existence of a unique solution for problem \eqref{S1}. In Theorem \ref{Local} we will prove the existence and the uniqueness of a local solution. Moreover, in Theorem \ref{global} the existence and the uniqueness of a global solution will be proved.\\
Before we will state the following lemma that will be used in the proof of the existence and uniqueness part of Theorem \ref{Local}. 
\begin{lemma}
Let $u_0\in L^\infty(\RN)$ and let $u\in C([0,T),L^\infty(\RN))$. Then $u$ is a solution for problem \eqref{S1} if and only if 
$$u(x,t)=S(t)u_0(x)-\int_0^tS(t-s)h(s)u^p(s)ds.$$
\end{lemma}
We will start by a result of local existence. 
\begin{theorem}[Local Existence]\label{Local}
Let $u_0\in L^\infty(\RN)$ and $p>1$. Then problem \eqref{S1} possesses a unique solution $u\in L^\infty(\RN)$ on the interval $[0,T]$, where $T=T(u_0)$. Moreover, there exist a maximal time $T_m\in (T,\infty)$ with the following properties. 
\begin{enumerate}
\item The solution $u$ can be continued in a unique way to $L^\infty(\RN)$ solution on the interval $[0,T_m)$. 
  \item If $T_m=\infty$, then $u$ is a global solution for \eqref{S1}.
  \item If $T_m<\infty$, then $\lim\limits_{t\to T_m}\|u(t)\|_{L^\infty(\RN)}=\infty$.
  \item If $u_0(x)\geq 0$, then $u(x,t)\geq 0$.
  \item If $u_0\in L^\infty(\RN)\cap L^r(\RN)$, then $u\in C\left([0,T_m),L^\infty(\RN)\cap L^r(\RN)\right)$
\end{enumerate}
\end{theorem}
\begin{proof}
we will proceed in six steps.\\
Step 1: Existence and uniqueness.\\
To prove the existence of a unique solution for the problem \eqref{S1}, a complete metric space $E_{T}$ should be constructed. To achieve this goal assume that first $\|u_0\|_\infty\ >0$ and let $k>1$, and define the space $E_T$ by
$$E_T=\{u\in L^\infty\left([0,T],L^\infty(\RN)\right) / \| u\|_1<k\|u_0\|_{L^\infty(\RN)}\}$$ 
equipped with the metric \\
$$d(u,v)=\|u-v\|_1,$$
where 
$$\|u\|_1=\sup_{t\in[0,T]}\|u\|_{L^\infty(\RN)}$$
It follows directly that $\left(E_T,d\right)$ is a complete metric space.\\
Now, For a fixed $u\in E_T$ and $0<s<t<T$ define the operator $\Phi_T$ by
$$\Phi_T(u)=S(t)u_0(x)-\int_0^tS(t-s)h(s)u^p(s)ds$$
However,
\begin{eqnarray*}
  \|u\|_1 &=& \left\|S(t)u_0(x)-\int_0^tS(t-s)h(s)u^p(s)ds\right\|_1, \\
   &\leq& \|S(t)u_0(x)\|_1+\left\|\int_0^tS(t-s)h(s)u^p(s)ds\right\|_1, \\
   &\leq& \|u_0(x)\|_{L^\infty(\RN)}+\left\|\left\|\int_0^tS(t-s)h(s)u^p(s)ds\right\|_{L^\infty(\RN)}\right\|_{L^\infty([0,T])}, \\
   &\leq& \|u_0(x)\|_{L^\infty(\RN)}+\left\|\int_0^t\left\|S(t-s)h(s)u^p(s)\right\|_{L^\infty(\RN)}ds\right\|_{L^\infty([0,T])},\\
   &\leq& \|u_0(x)\|_{L^\infty(\RN)}+\left\|\int_0^t h(s)\|S(t-s)\|_{L^\infty(\RN)}\|u^p(s)\|_{L^\infty(\RN)}ds\right\|_{L^\infty([0,T])}, \\
   &\leq& \|u_0(x)\|_{L^\infty(\RN)}+TM\|u\|^p_1, \\
   &\leq& \|u_0(x)\|_{L^\infty(\RN)}+TMk^p\|u_0\|^p_{L^\infty(\RN)}, \\
   &=& \left(1+MTk^p\|u_0\|^{p-1}_{L^\infty(\RN)}\right)\|u_0\|_{L^\infty(\RN)},
\end{eqnarray*}
where $M=\|h(t)\|_{L^\infty([0,T])}$\\
As a result of that, $\Phi_T: E_T\mapsto E_T$ if and only if $T\leq\frac{k-1}{Mk^p\|u_0\|^{p-1}_{L^\infty(\RN)}}$\\
Adding to that, for $u, v\in E_T$ we have,
\begin{eqnarray*}
  \|\Phi_T(u)-\Phi_T(v)\|_1 &=& \left\|\int_0^tS(t-s)h(s)v^p(s)ds-\int_0^tS(t-s)h(s)u^p(s)ds\right\|_1, \\
    &\leq& \left\|\left\|\int_0^tS(t-s)h(s)\left(|v|^{p-1}(s)v(s)-|u|^{p-1}(s)u(s)\right)ds\right\|_{L^\infty(\RN)}\right\|_{L^\infty([0,T])}, \\
    &\leq& \left\|\left\|\int_0^TS(t-s)h(s)\left(|v|^{p-1}(s)v(s)-|u|^{p-1}(s)u(s)\right)ds\right\|_{L^\infty(\RN)}\right\|_{L^\infty([0,T])},  \\
    &\leq& CMT\|u(s)-v(s)\|_1,
\end{eqnarray*}
where $C=\max\{\|v\|_1^{p-1},\|u\|_1^{p-1}\}$.\\
Following direct simplifications will have $CMT\leq \frac{k-1}{k}<1$.\\
It follows directly that $\Phi_T$ is a contraction on $E_T$.\\
Therefore, by applying the Banach fixed point theorem, it follows directly that the equation $\Phi_T(u)=u$ has a unique solution in $E_T$.\\
Step 2. Regularity.\\
Consider the map 
$$f:E_T\mapsto L^1([0,T],L^\infty(\RN))$$  $$u\mapsto f(u)=h(t)u^p.$$
It is clear that $f$ is Lipschitz continuous and that $u^p\in L^1([0,T],L^\infty(\RN))$.\\
As a result of that $u\in C([0,T],L^\infty(\RN))$ and therefore,
\begin{equation}\label{Mild}
u(x,t)=S(t)u_0(x)-\int_0^tS(t-s)h(s)u^p(s)ds,
\end{equation}
is a unique mild solution for \eqref{S1}. \\
Moreover,\\
if $u_0\in L^\infty(\RN)\cap L^r(\RN)$, for $r>n(p-1)/2$ define 
$$E^r_T=\{u\in C\left([0,T_m],L^\infty(\RN)\cap L^r(\RN)\right) / \| u\|_1<k\|u_0\|_{L^\infty(\RN)}\}$$ 
and repeat the same process and in Step 1 will have the desired result.\\
Define now, 
$$T_m=\sup\{T>0\, /\, u\in E_T\,\;\text{is a mild solution for problem \eqref{S1}}\}\leq\infty$$ 
Step 3. The alternative property.\\
for $0<t<\tau <T_m$
\begin{eqnarray*}
  u(t+\tau) &=& S(t+\tau)u_0(x)-\int_0^{t+\tau}S(t+\tau-s)h(s)u^p(s)ds \\
   &=& S(t)S(\tau)u_0(x)-\int_0^tS(t+\tau-s)h(s)u^p(s)ds-\int_t^\tau S(t+\tau-s)h(s)u^p(s)ds \\
   &=& S(\tau)\left(S(t)S(\tau)u_0(x)-\int_0^tS(t-s)h(s)u^p(s)d\right)-\int_t^\tau S(t+\tau-s)h(s)u^p(s)ds \\
   &=& S(\tau)u(t)-\int_t^\tau S(t+\tau-s)h(s)u^p(s)ds
\end{eqnarray*}
Step 4. Continuous dependence.\\
Let us denote by $U(t)u_0$ the solution constructed in step 1 for the problem \eqref{S1} and let $u_0, v_0 \in E_T$ it follows that for a certain choice for $T$,
\begin{eqnarray*}
  \|U(t)u_0-U(t)v_0\|_{L^\infty(\RN)}&=&\left\|S(t)(u_0-v_0)-\int_0^tS(t-s)h(s)(u^p(s)-v^p(s))ds\right\|_{L^\infty(\RN)},\\
   &\leq&\|u_0-v_0\|{L^\infty(\RN)}+\\
   &&\int_0^t h(s)(\left\|(|u(s)|^{p-1}v(s)-|v(s)|^{p-1}v(s))\right\|_{L^\infty(\RN)}ds,  \\
  &\leq&\|u_0-v_0\|_{L^\infty(\RN)}+CMTk^{p-1}\|u_0-v_0\|_{L^\infty(\RN)},\\
  &\leq& (1+CMTk^{p-1})\|u_0-v_0\|_{L^\infty(\RN)},
\end{eqnarray*}
consequently, $U(t)u_0:L^\infty(\RN)\to L^\infty(\RN)$ is a Lipschitz continuous map. 
Step 5. For the positivity.\\
Since $u_0(x)\geq 0$, if follows directly from the maximum principle that
 $$u(x,t)\geq 0.$$
Step 6. The maximal existence time (blow-up).\\
Assume that $T_m$ is finite and define $w(s)$ by
$$w(s)=\left(\frac{(k-1)(p-1)}{Mk^p}\right)^{-1}T(t-s)u(s),$$
it follows directly that \\
\begin{equation}\label{w}
\begin{split}
&\|w(s)\|_{L^\infty(\RN)}\leq \left(\frac{(k-1)(p-1)}{Mk^p}\right)^{-1}\|u(s)\|_{L^\infty(\RN)}\\
&\text{and}\quad\frac{d}{ds}w(s)=\left(\frac{(k-1)(p-1)}{Mk^p}\right)^{-1}u^p(s).
\end{split}
\end{equation}
Now, let $v(s)=\|w(s)\|_{L^\infty(\RN)}$ and consider the system
\begin{equation}\label{blow}
\left\{
\begin{split}
&v_t\geq \left(\frac{(k-1)(p-1)}{Mk^p}\right)^{-1}v^p(t),\\
&v(0)=\|u_0\|_{L^\infty(\RN)}
\end{split}
\right.
\end{equation}
thus
$$\frac{dv}{v^p}\geq\left(\frac{(k-1)(p-1)}{Mk^p}\right)^{-1},$$
so,
$$\frac{v^{1-p}}{1-p}\geq\left(\frac{(k-1)(p-1)}{Mk^p}\right)^{-1}t+\frac{\|u_0\|^{1-p}_{L^\infty(\RN)}}{1-p},$$
therefore,
\begin{eqnarray}\label{v}
v(t)^{1-p}&\leq&\left((1-p)\left(\frac{(k-1)(p-1)}{Mk^p}\right)^{-1}t+\|u_0\|^{1-p}_{L^\infty(\RN)}\right),\\
\frac{1}{v(t)^{p-1}}&\leq&\left(-\left(\frac{(k-1)}{Mk^p}\right)^{-1}t+\|u_0\|^{1-p}_{L^\infty(\RN)}\right),\\
v(t)&\geq& \left(\frac{1}{-\left(\frac{(k-1)}{Mk^p}\right)^{-1}t+\|u_0\|^{1-p}_{L^\infty(\RN)}}\right)^{p-1}
\end{eqnarray}
and its clear that $$\lim\limits_{t\to T}v(t)=\infty,$$
where $T=\frac{(k-1)}{Mk^p\|u_0\|^{p-1}_{L^\infty(\RN)}}$.\\ 
Finally, using \eqref{w} we will have 
$$\lim\limits_{t\to T}\|u(t)\|_{L^\infty(\RN)}=\infty.$$ 
\end{proof}
Now, we will state the global existence result.
\begin{theorem}[Global Existence]\label{global}
For all $u_0\in L^\infty(\RN)$ the system \eqref{S1} has a global solution 
\end{theorem}
\begin{proof}
As we have 
$$u(x,t)=S(t)u_0(x)-\int_0^tS(t-s)h(s)u^p(s)ds$$
and since, $u_0(x)\geq 0$ it follows directly that $$\|u(x,t)\|_{L^\infty(\RN)}\leq\|u_0\|_{L^\infty(\RN)}<\infty.$$
Thus, the proof is completed
\begin{remark}
Since, $u_0\in L^\infty(\RN)$ and $u\in E_T$ it follows directly that the solution  constructed in Theorem \ref{Local} is in $B_{k\|u_0\|_{L^\infty(\RN)}}$ where $B_{k\|u_0\|_{L^\infty(\RN)}}$ is the ball in $E_T$ of radius $R=k\|u_0\|_{L^\infty(\RN)}$.        
\end{remark}
\begin{remark}
To prove $u\in C([0,\infty),L^\infty(\RN))$. It is clear that for $\varepsilon >0$ we have,
\begin{equation}
u(x,t+\varepsilon)=S(t+\varepsilon)u_0(x)-\int_0^{t+\varepsilon} S(t+\varepsilon-s)h(s)u^p(s)ds,
\end{equation}  
since, $u_0\in L^\infty(\RN)$ thus, $S(t+\varepsilon)u_0(x)\in C([0,\infty),L^\infty(\RN))$.
Moreover, $S(t+\varepsilon-s)h(s)u^p(s)$ is bounded and converges to $S(t-s)h(s)u^p(s)$ as $\varepsilon\to 0$ in $L^\infty(\RN)$.\\
Thus, by using the Lebesque dominated convergence theorem,
$$\int_0^{t+\varepsilon} S(t+\varepsilon-s)h(s)u^p(s)ds\to\int^t_0 S(t-s)h(s)u^p(s).$$
As a result of that 
$$u(x,t+\varepsilon)\to u(x,t),$$
as $\varepsilon\to 0$ in $L^\infty(\RN)$. Thus, the desired result is obtained.
\end{remark}
\begin{remark}
The result of Theorem \ref{global} can be obtained in another way. Assume to the contrary that $T_m$ is finite, and thus by Theorem \ref{Local} 
$$\lim\limits_{t\to T_m}\|u(t)\|_{L^\infty(\RN)}=\infty.$$
First, note that function $f(u)=u^p$ is Lipschitz continuous on $E_T$ therefore, there exist $C>0$ such that 
$$|f(u)|\leq C|u|,$$ 
and thus 
$$\|f(u(t))\|_{L^\infty(\RN)}\leq C\|u(t)\|_{L^\infty(\RN)},$$
using the solution constructed in step 1 of Theorem \ref{Local} we will have
\begin{eqnarray*}
\|u(t)\|_{L^\infty(\RN)}&=&\|S(t)u_0\|_{L^\infty(\RN)}+\left\|\int_0^1h(s)S(t-s)f(u)ds\right\|_{L^\infty(\RN)},\\
&\leq& \|u_0\|_{L^\infty(\RN)}+CM\int_0^1\|u(t)\|_{L^\infty(\RN)}ds.
\end{eqnarray*}
Applying Gronwall's lemma it follows that
$$\|u(t)\|_{L^\infty(\RN)}\leq \left(C+\|u_0\|_{L^\infty(\RN)}\right)e^{Ct}.$$
Consequently,
$$\lim\limits_{t\to T_m}\|u(t)\|_{L^\infty(\RN)}\leq \left(C+\|u_0\|_{L^\infty(\RN)}\right)e^{CT_m} $$
which is a contradiction. Finally $T_m=\infty$. 
\end{remark}
\end{proof}
\section{Regularity}
We proved that problem \eqref{S1} has a unique solution over $[0,T)\times\RN$. Also, either $T$ is finite and in this case $\lim\limits_{t\to T}\|u(t)\|_{L^\infty(\RN)}=\infty$ or $T=\infty$ and in this case $u$ is a global solution. Moreover, it is clear that the regularity of the solution $u$ depends on the regularity of the initial condition $u_0$. In Corollaries \ref{C1} and \ref{C2} we will state the minimum regularity required for $u_0$ to obtain a classical solution.\\
Before, the following lemma is an essential tool for proving the required results
\begin{lemma}[Schauder Regularity]
Consider the problem
\begin{equation}\label{S2}
\left\{
\begin{split}
&u_t=-t^\beta\LL u-f(u),\quad (x,t)\in(0,\infty)\times\RN,\\
&u(x,0)=u_0(x), \quad x\in\RN,
\end{split}
\right.
\end{equation}
and assume that
\begin{equation}
\begin{split}
 &f\in C^{\alpha,\alpha/2}([0,\infty),L^\infty(\RN)),\\ 
 &\text{and}\quad u_0\in C^{\alpha}(\RN).
 \end{split}
\end{equation}
Then
$$u\in C^{2+\alpha,1+\alpha/2}([0,\infty),L^\infty(\RN))\cap C^{2s+\alpha,1+\alpha/2}([0,\infty),L^\infty(\RN)),$$
satisfying the following smoothing estimate
\begin{equation}
\|u\|_{C^{2+\alpha,1+\alpha/2}([0,\infty),L^\infty(\RN))}\\
\leq C\left(\|u_0\|_{C^{\alpha}(L^\infty(\RN))}+\|f(u)\|_{C^{\alpha,\alpha/2}([0,\infty),L^\infty(\RN))}\right).
\end{equation}
\end{lemma}
\begin{corollary}\label{C1}
Let $u_0\in C_0([0,\infty),L^\infty(\RN))$. Then problem \eqref{S1} admits a unique classical (global) $L^\infty(\RN)$ solution.
\end{corollary}
\begin{proof}
If we assume that $u_0\in C_0([0,\infty),L^\infty(\RN))$, then its clear that 
$$u(t)\in C_0([0,\infty),L^\infty(\RN)).$$
Moreover, let $\phi\in D([0,\infty))$ be a compactly supported function and let $v=\phi u$.\\
It follows directly that $v_t=\phi_t u+u_t\phi$ and that $v$ solves the problem
\begin{equation}\label{classical}
\left\{
\begin{split}
&v_t+\Delta v=\widetilde{f},\quad (x,t)\in [0,\infty)\times(\RN),\\
&v_0=\phi(0)u_0, \quad x\in\RN,
\end{split}
\right.
\end{equation} 
where $\widetilde{f}=\phi_t u+\phi u^p$.\\
Observe that $\widetilde{f}$ is Hölder continuous and thus $\widetilde{f}\in C^{\alpha,\alpha/2}([0,\infty)\times\RN)$ and that $v(0)\in C^{\alpha}(\RN)$, for $\alpha\in (0,1)$. Using Schauder regularity theory for parabolic equations we obtain that \eqref{classical} admits a unique classical solution $v\in L^\infty(\RN)$.\\
Consequently, $u\in W^{2,1,\infty}((0,\infty)\times\RN)$.\\
Moreover, using the Sobolev embedding theorem we have 
$$W^{2,1,\infty}([0,\infty)\times\RN)\hookrightarrow C^{\alpha,\alpha/2}([0,\infty)\times\RN))$$ 
which gives $f(u)=u^p$ is Hölder continuous, and as a result of that 
$$u\in C^{2+\alpha,1+\alpha/2}([0,\infty)\times\RN),$$ 
is a classical $L^\infty(\RN)$ solution for the problem \eqref{S1}
\end{proof}
\begin{corollary}\label{C2}
Let $u_0\in C_0([0,\infty),L^\infty(\RN)\cap L^r(\RN))$ and $r>n(p-1)/2$. Then problem \eqref{S1} admits a unique classical (global) $L^r(\RN)$ solution.
\end{corollary}
\begin{proof}
Let $u_0\in C_0([0,\infty),L^\infty(\RN)\cap L^r(\RN))$, then its clear that $$u(t)\in C_0([0,\infty),L^\infty(\RN)\cap L^r(\RN)),$$ by repeating the same argument as in step 1 of Theorem \eqref{Local}.\\
Moreover, if we choose $\phi\in D([0,\infty))$ be a compactly supported function and by taking the change of variable $v=\phi u$.\\
It follows directly that $v_t=\phi_t u+u_t\phi$ and that $v$ solves the problem
\begin{equation}\label{classicalv}
\left\{
\begin{split}
&v_t+\Delta v=\widetilde{f},\quad (x,t)\in [0,\infty)\times(\RN),\\
&v_0=\phi(0)u_0, \quad x\in\RN,
\end{split}
\right.
\end{equation} 
where $\widetilde{f}=\phi_t u+\phi u^p$.\\
Since, $\widetilde{f}$ is Hölder continuous thus, $\widetilde{f}\in C^{\alpha,\alpha/2}([0,\infty)\times\RN)$ and $v(0)\in C^{\alpha}(\RN)$, so, by using Schauder regularity theory for parabolic equations we obtain that \eqref{classicalv} admits a unique classical solution $v\in L^\infty(\RN)$.\\
Consequently, $u\in W^{2,1,r}((0,\infty)\times\RN)$.\\
Moreover, using the Sobolev embedding theorem we have 
$$W^{2,1,r}([0,\infty)\times\RN)\hookrightarrow C^{\alpha,\alpha/2}([0,\infty)\times\RN)),$$
for $\alpha=1-\frac{N+2}{r}$, which gives $f(u)=u^p$ is Hölder continuous, and as a result of that 
$$u\in C^{2+\alpha,1+\alpha/2}([0,\infty)\times\RN),$$ 
is a classical $L^\infty(\RN)\cap L^r(\RN)$ solution for the problem \eqref{S1}
\end{proof}
we will end now, with the following remarks.
\begin{remark}
In the proof of Corollary \ref{C1} we stated that 
$$W^{2,1,\infty}([0,\infty)\times\RN)\hookrightarrow C^{\alpha,\alpha/2}([0,\infty)\times\RN)).$$
In this case $\alpha=1$.\\
Moreover, in the proof of Corlloaries \ref{C1} and \ref{C2} we stated that 
$$f(u)=u^p$$
is Hölder continuous the justification of this assumption is as follows, 
\begin{eqnarray}
\|u^p-v^p\|_1&\leq& C_1\|u-v\|_1,\\
&\leq& C_1C_2\left(|x_1-x_2|^\alpha+|t_1-t_2|^{\alpha/2}\right),
\end{eqnarray}   
for $\alpha\in(0,1)$, $C_1$ is the Lipschitz constant and $C_2$ is the Hölder constant associated to $u$.
\end{remark}

\end{document}